\documentclass{amsart}

\usepackage{amssymb,amsmath,amsthm,latexsym,stmaryrd}
\usepackage[pagebackref=true,colorlinks=true]{hyperref}
\usepackage{color}

\numberwithin{equation}{section}
\newtheorem{thm}{Theorem}[section]
\newtheorem{prop}[thm]{Proposition}
\newtheorem{cor}[thm]{Corollary}
\newtheorem{lem}[thm]{Lemma}
\newtheorem{defn}[thm]{Definition}

\newcommand{\thmref}[1]{Theorem~\ref{#1}}
\newcommand{\propref}[1]{Proposition~\ref{#1}}
\newcommand{\lemref}[1]{Lemma~\ref{#1}}
\newcommand{\corref}[1]{Corollary~\ref{#1}}

\newcommand{\ie}{i.\hspace{.5pt}e.\ }
\newcommand{\tr}{\operatorname{tr}}

\newcommand{\R}{\mathbb{R}}
\newcommand{\M}{(\mathcal{M},\allowbreak{}\phi,\allowbreak{}\xi,\allowbreak{}\eta,g)}
\newcommand{\MM}{\mathcal{M}}

\newcommand{\Df}{\dot{D}}
\newcommand{\Tf}{\dot{T}}
\newcommand{\Rf}{\dot{R}}
\newcommand{\Qf}{\dot{Q}}

\newcommand{\HH}{\mathcal{H}}
\newcommand{\VV}{\mathcal{V}}
\newcommand{\GG}{(\mathcal{G},\ff,\allowbreak{}\xi,\allowbreak{}\eta,g)}
\newcommand{\Div}{\operatorname{div}}
\newcommand{\Span}{{\rm span}}

\newcommand{\ff}{\phi}
\newcommand{\F}{\mathcal{F}}

\newcommand{\n}{\nabla}
\newcommand{\N}{\widehat{N}}

\newcommand{\ta}{\theta}

\newcommand{\lm}{\lambda}

\newcommand{\om}{\omega}
\newcommand{\D}{\mathrm{d}}

\newcommand{\g}{\tilde{g}}
\newcommand{\sx}{\mathop{\mathfrak{S}}}

\begin{document}

\vspace{2cm}

\title[First natural connection on Riemannian $\Pi$-manifolds]
{First natural connection on Riemannian $\Pi$-manifolds}

\author{Hristo Manev}
\address{Medical University of Plovdiv, Faculty of Pharmacy,
Department of Medical Physics and Biophysics,   15-A Vasil Aprilov
Blvd.,   Plovdiv 4002,   Bulgaria;}
\email{hristo.manev@mu-plovdiv.bg}



\begin{abstract}
A natural connection with torsion is defined and it is called the first natural connection on Riemannian $\Pi$-manifold. Relations between the introduced connection and the Levi-Civita connection are obtained, as well as relations between their respective curvature tensors, torsion tensors, Ricci tensors, and scalar curvatures in the main classes of a classification of  Riemannian $\Pi$-manifolds are presented. An explicit example of dimension 5 is provided.
\end{abstract}
\maketitle

\section{Introduction}\label{sect-0}

In the present work we study the differential geometry of the almost paracontact almost paracomplex Riemannian manifolds, called briefly Riemannian $\Pi$-manifolds \cite{ManSta01,ManVes18}. The considered odd dimensional manifolds have traceless induced almost product structure on the paracontact distribution and the restriction on the paracontact distribution of the almost paracontact structure is an almost paracomplex structure. The start of the investigation of the Riemannian $\Pi$-manifolds is given in \cite{ManSta01} by the name almost paracontact Riemannian manifolds of type $(n,n)$. After that their study continued in series of works (e.g. \cite{ManVes18,IvMan2,HMan3,HM17}).

In \cite{ManSta01}, M. Manev and M. Staikova presented a classification of the Riemannian $\Pi$-ma\-ni\-folds with respect to the fundamental tensor $F$ which contains eleven basic classes. We consider four of these eleven basic classes, the so-called main classes, in which $F$ is expressed explicitly by the metrics and the Lee forms.

In differential geometry of manifolds with additional tensor structures important role play those affine connections which preserve the structure tensors and the metric, known also as natural connections (e.g. \cite{KobNom,Ale-Gan2,Gan-Mi,Mek-P-con,Man-Gri2,StaGri}).
We define a non-symmetric natural connection and we call it first natural connection on Riemannian $\Pi$-manifold. We obtain relations between the introduced connection and the Levi-Civita connection, as well as we study some of its curvature characteristics in the main classes.

The paper is structured as follows.
After this introductory Section 1, in Section 2, we recall some preliminary background facts about the considered geometry. In the next Section 3, we define the concept of natural connection on Riemannian $\Pi$-manifold and we prove a necessary and sufficient condition an affine connection to be natural.
Section 4 is devoted to the first natural connection on Riemannian $\Pi$-manifold and its relations with the Levi-Civita connection. Moreover, in this section we prove assertions for relations between these two connections and their respective curvature tensors, torsion tensors, Ricci tensors, and scalar curvatures.
In the final Section 5, we support the results made by an explicit example of dimension 5.

\section{Riemannian $\Pi$-Manifolds}\label{sect-1}

Let $\M$ be a {Riemannian $\Pi$-manifold}, where $\mathcal{M}$ is $(2n+1)$-di\-men\-sional differentiable manifold, equipped with a Rie\-mannian metric $g$ and a Riemannian $\Pi$-structure $(\ff,\xi,\eta)$. This structure consists a (1,1)-tensor field $\ff$, a Reeb vector field $\xi$ and its dual 1-form $\eta$. The following basic identities and their immediately derived properties are valid:
\begin{equation}\label{strM}
\begin{array}{c}
\ff\xi = 0,\qquad \ff^2 = I - \eta \otimes \xi,\qquad
\eta\circ\ff=0,\qquad \eta(\xi)=1,\\ \
\tr \ff=0,\qquad g(\ff x, \ff y) = g(x,y) - \eta(x)\eta(y),
\end{array}
\end{equation}
\begin{equation}\label{strM2}
\begin{array}{ll}
g(\ff x, y) = g(x,\ff y),\qquad &g(x, \xi) = \eta(x),
\\
g(\xi, \xi) = 1,\qquad &\eta(\n_x \xi) = 0,
\end{array}
\end{equation}
where $I$ and $\n$ denote the identity transformation on $T\mathcal{M}$ and the Levi-Civita connection of $g$, respectively (\cite{Sato76,ManVes18}).
Here and further, $x$, $y$, $z$, $w$ stand for arbitrary differentiable vector fields on $\mathcal{M}$ or tangent vectors at a point of $\MM$.

The associated metric $\g$ of $g$ on $\M$ is defined by $\g(x,y)=g(x,\ff y)+\eta(x)\eta(y)$. It is an indefinite metric of signature $(n + 1, n)$ and it is compatible with the manifold in the same way as $g$. In further investigations we use the following notations:
\begin{equation}\label{g***}
g^*(x,y)=g(x,\ff y), \qquad g^{**}(x,y)=g(\ff x,\ff y).
\end{equation}

Using $\xi$ and $\eta$ on an arbitrary Riemannian $\Pi$-manifold $\M$, we consider two complementary distributions of $T\MM$---the horizontal distribution $\HH=\ker(\eta)$ and the vertical distribution $\VV=\Span(\xi)$. They are mutually orthogonal with respect to the both metrics $g$ and $\tilde{g}$, \ie
\begin{equation}\label{HHVV}
\HH\oplus\VV =T\MM,\qquad \HH\;\bot\;\VV,\qquad \HH\cap\VV=\{o\},
\end{equation}
where $o$ stands for the zero vector field on $\MM$.
In this way the respective horizontal and vertical projectors are determined by $\mathrm{h}:T\MM\mapsto\HH$ and $\mathrm{v}:T\MM\mapsto\VV$.

An arbitrary vector field $x$ has corresponding projections $x^{\mathrm{h}}$ and $x^{\mathrm{v}}$ so that
\begin{equation}\label{hv}
x=x^{\mathrm{h}}+x^{\mathrm{v}},
\end{equation}
where
\begin{equation}\label{Xhv}
x^{\mathrm{h}}=\ff^2x, \qquad x^{\mathrm{v}}=\eta(x)\xi
\end{equation}
are the so-called horizontal and vertical component of $x$, respectively.

Let us denote by $\nabla$ the Levi-Civita connection of $g$. The following tensor field $F$ of type $(0,3)$ plays an important role in the geometry of the Riemannian $\Pi$-ma\-ni\-folds: \cite{ManSta01}
\begin{equation}\label{F}
F(x,y,z)=g\bigl( \left( \nabla_x \ff \right)y,z\bigr).
\end{equation}

From \eqref{strM} and \eqref{F} the following general properties of $F$ are obtained: \cite{ManSta01}
\begin{equation}\label{F-prop}
\begin{array}{l}
F(x,y,z)=F(x,z,y)=-F(x,\ff y,\ff z) +\eta(y)F(x,\xi,z)+\eta(z)F(x,y,\xi),\\[6pt]
F(x,y,\ff z)=-F(x,\ff y, z)+\eta(z)F(x,\ff y,\xi) +\eta(y)F(x,\ff z,\xi),\\[6pt]
F(x,\ff y,\ff z)=-F(x,\ff^2 y,\ff^2 z),\\[6pt]
F(x,\ff y,\ff ^2 z)=-F(x,\ff^2 y,\ff z).
\end{array}
\end{equation}

\begin{lem}[\cite{ManVes18}]\label{lem-F}
The following identities are valid:
\begin{enumerate}
\item[$1)$] $(\nabla_x \eta)(y)=g( \nabla_x \xi,y)$,
\item[$2)$] $\eta(\nabla_x \xi)=0$,
\item[$3)$] $F(x,\ff y,\xi)=-(\nabla_x \eta)(y)$.
\end{enumerate}
\end{lem}

The 1-forms associated with $F$, known as Lee forms, are defined by:
\begin{equation*}\label{t}
\theta=g^{ij}F(e_i,e_j,\cdot),\quad
\theta^*=g^{ij}F(e_i,\ff e_j,\cdot), \quad \omega=F(\xi,\xi,\cdot),
\end{equation*}
where $\left(g^{ij}\right)$ is the inverse matrix of $\left(g_{ij}\right)$ of $g$ with respect to
a basis $\left\{\xi;e_i\right\}$ of $T_p\mathcal{M}$ $(i=1,2,\dots,2n; p\in \mathcal{M})$.
Using \eqref{F-prop}, the following relations for the Lee forms are obtained: \cite{ManSta01}
\begin{equation}\label{ta-prop}
\begin{array}{l}
\om(\xi)=0,\qquad \ta^*\circ\ff=-\ta\circ\ff^2,\qquad \ta^*\circ\ff^2=\ta\circ\ff.
\end{array}
\end{equation}

In \cite{ManSta01}, M. Manev and M. Staikova presented a classification of Riemannian $\Pi$-ma\-ni\-folds with respect to the fundamental tensor $F$ which contains eleven basic classes denoted by $\F_1$, $\F_2$, $\dots$, $\F_{11}$. The intersection of the basic classes is the special class $\F_0$ determined by the condition $F=0$. Let us remark that the main object of our considerations are the so-called main classes of the considered manifolds among the basic eleven. These are the classes $\F_1$, $\F_4$, $\F_5$, $\F_{11}$ in which the fundamental tensor $F$ is expressed explicitly by the metrics and the Lee forms. The characteristic conditions of these classes are: \cite{ManSta01,ManVes18}
\begin{equation}\label{Fcon}
\begin{split}
\F_{1}:\quad &F(x,y,z)=\dfrac{1}{2n}\bigl\{g(\ff x,\ff y)\ta(\ff^2 z)+g(\ff x,\ff z)\ta(\ff^2 y)\\
&\phantom{F(x,y,z)=\dfrac{1}{2n}\bigl\{}-g( x,\ff y)\ta(\ff z)-g( x,\ff z)\ta(\ff y)\bigr\};\\[6pt]
\F_{4}:\quad &F(x,y,z)=\dfrac{\ta(\xi)}{2n}\bigl\{g(\ff x,\ff y)\eta(z)+g(\ff x,\ff z)\eta(y)\bigr\};\\[6pt]
\F_{5}:\quad &F(x,y,z)=\dfrac{\ta^*(\xi)}{2n}\bigl\{g( x,\ff y)\eta(z)+g(x,\ff z)\eta(y)\bigr\};\\[6pt]
\F_{11}:\quad &F(x,y,z)=\eta(x)\left\{\eta(y)\om(z)+\eta(z)\om(y)\right\}.
\end{split}
\end{equation}

The $(1,2)$-tensors $N$ and $\N$ defined by
\begin{equation*}\label{N-nff}
\begin{split}
N(x,y)=&\left(\n_{\ff x}\ff\right)y-\ff\left(\n_{x}\ff\right)y-\left(\n_{x}\eta\right)(y)\xi\\[6pt]
& -\left(\n_{\ff
y}\ff\right)x+\ff\left(\n_{y}\ff\right)x+\left(\n_{y}\eta\right)(x)\xi,
\end{split}
\end{equation*}
\begin{equation*}\label{N1-nff}
\begin{split}
\N(x,y)=&\left(\n_{\ff x}\ff\right)y-\ff\left(\n_{x}\ff\right)y-\left(\n_{x}\eta\right)(y)\xi\\[6pt]
& +\left(\n_{\ff
y}\ff\right)x-\ff\left(\n_{y}\ff\right)x-\left(\n_{y}\eta\right)(x)\xi
\end{split}
\end{equation*}
are called Nijenhuis tensor and associated Nijenhuis tensor, respectively, for the $\Pi$-structure on $\MM$ \cite{ManVes18}.

It can be immediately established that we have an antisymmetric tensor $N$ and a symmetric $\N$, \ie
\begin{equation}\label{NN-prop}
\begin{array}{l}
N(x,y) = -N(y,x),\qquad
\N(x, y) = \N(y, x).
\end{array}
\end{equation}

The corresponding $(0,3)$-tensors of $N$ and $\N$ on $\M$ are denoted by the same letter and are expressed by means of $F$ through the equalities: \cite{ManVes18}
\begin{equation*}\label{NN1-F}
\begin{array}{ll}
N(x,y,z)=g\left(N(x,y),z\right)\\[6pt]
\phantom{N(x,y,z)}=F(\ff x,y,z)-F(\ff y,x,z)-F(x,y,\ff z)+F(y,x,\ff z)\\[6pt]
\phantom{N(x,y,z)=}+\eta(z)\left\{F(x,\ff y,\xi)-F(y,\ff
x,\xi)\right\},\\[6pt]
\N(x,y,z)=g\left(\N(x,y),z\right)\\[6pt]
\phantom{\N(x,y,z)}=F(\ff x,y,z)+F(\ff y,x,z)-F(x,y,\ff z)-F(y,x,\ff z)\\[6pt]
\phantom{\N(x,y,z)=}+\eta(z)\left\{F(x,\ff y,\xi)+F(y,\ff x,\xi)\right\}.
\end{array}
\end{equation*}

On the other hand the fundamental tensor $F$ of a Riemannian $\Pi$-manifold can be expressed only by the pair of tensors $N$ and $\N$ as follows: \cite{ManVes18}
\begin{equation}\label{F=NN}
\begin{array}{l}
F(x,y,z)=\dfrac14\bigl\{N(\ff x,y,z)+N(\ff x,z,y)+\N(\ff x,y,z)+\N(\ff x,z,y)\bigr\}\\[6pt]
\phantom{F(x,y,z)=}-\dfrac12\eta(x)\bigl\{N(\xi,y,\ff z)+\N(\xi,y,\ff z)+\eta(z)\N(\xi,\xi,\ff y)\bigr\}.
\end{array}
\end{equation}

Let $R$ denote the {curvature tensor} of type $(1,3)$ for the Levi-Civita connection $\n$ generated by the metric $g$ on $\M$, \ie
\begin{equation}\label{R_xyz}
R(x,y)z=\n_x \n_y z-\n_y \n_x z-\n_{[x,y]} z.
\end{equation}
Let us denote the corresponding curvature $(0,4)$-tensor by the same letter and let us define it by the following equality:
\begin{equation}\label{R_xyzw}
R(x,y,z,w)=g(R(x,y)z,w).
\end{equation}
The following known basic properties hold for $R$:
\begin{eqnarray}
  & R(x,y,z,w)=-R(y,x,z,w)=-R(x,y,w,z), \label{R-prop-1}
  \\[6pt]
  & R(x,y,z,w)+R(y,z,x,w)+R(z,x,y,w)=0\label{R-prop-2}.
\end{eqnarray}

For $R$ we define {Ricci tensor} $\rho$ of type $(0,2)$ as follows
\begin{equation}\label{ro}
\rho(x,y)=g^{ij}R(e_{i},x,y,e_{j}),
\end{equation}
and {scalar curvature} $\tau$ as the trace of $\rho$ through
\begin{equation}\label{tau}
\tau=g^{ij}\rho(e_{i},e_{j}).
\end{equation}

The {associated quantities} $\rho^*$ and $\tau^*$ corresponding to $\rho$ and $\tau$ are determined by the following equalities:
\begin{equation}\label{ro*-tau*}
\begin{array}{l}
\rho^*(x,y)=g^{ij}R(e_{i},x,y,\ff e_{j}),\qquad
\tau^*=g^{ij}\rho^*(e_{i},e_{j}).
\end{array}
\end{equation}

The notation $S \owedge P$ stands for the Kulkarni-Nomizu product of two tensors $S$ and $P$ of type (0,2), defined as follows:
\begin{equation}\label{KulNom}
\begin{split}
\left(S\owedge P\right)(x,y,z,w)=\,&S(x,z)P(y,w)-S(y,z)P(x,w)\\[6pt]
&+S(y,w)P(x,z)-S(x,w)P(y,z).
\end{split}
\end{equation}
It is easy to see that $S\owedge P$ possesses the basic properties \eqref{R-prop-1} and \eqref{R-prop-2} of $R$ just when $S$ and $P$ are symmetric tensors.

Let $T$ denote the {torsion tensor} of an arbitrary affine connection $D$, \ie
\begin{equation}\label{T-def}
T(x,y)=D_xy-D_yx-[x,y].
\end{equation}

Let us remark that $D$ is symmetric if and only if its torsion tensor $T$ is zero.

Let us denote by the same letter the corresponding $(0,3)$-tensor with respect to the metric $g$, \ie
\begin{equation}\label{D-T-03}
\begin{array}{l}
T(x,y,z)=g(T(x,y),z).
\end{array}
\end{equation}

{Torsion forms} $t$, $t^*$ and $\hat{t}$ of $T$ we call the associated 1-forms of $T$ defined by:
\begin{equation}\label{t}
\begin{array}{l}
t(x)=g^{ij}T(x,e_i,e_j),\qquad t^*(x)=g^{ij}T(x,e_i,\ff e_j),\qquad \hat{t}(x)=T(x,\xi,\xi)
\end{array}
\end{equation}
with with respect to a basis $\left\{\xi;e_i\right\}$ of $T_p\mathcal{M}$ $(i=1,2,\dots,2n; p\in \mathcal{M})$. Obviously, the identity $\hat{t}(\xi)=0$ holds.

\section{Natural connection on Riemannian $\Pi$-Manifolds}\label{sect-2}

Let us consider an arbitrary Riemannian $\Pi$-manifold $\M$.

\begin{defn}
An affine connection $D$ on a Riemannian $\Pi$-manifold $\M$ is called a {natural connection} for the Riemannian $\Pi$-structure $(\ff,\xi,\eta,g)$ if this structure is parallel with respect to $D$, \ie
\[
D\ff=D\xi=D\eta=Dg=0.
\]
\end{defn}
It is easily verified as a consequence that the associated metric $\tilde{g}$ is also parallel with respect to the natural connection $D$ on $\M$, \ie $D\tilde{g}=0$.

Therefore, $D$ on a Riemannian $\Pi$-manifold $\M\notin\F_0$ plays the same role as $\n$ on $\M\in\F_0$. Obviously, $D$ and $\n$ coincide when $\M\in\F_0$.

Let $Q$ denote the difference of $D$ and $\n$ which we call the {potential} of $D$ with respect to $\n$. Then we have
\begin{equation}\label{1}
D_xy=\n_xy+Q(x,y).
\end{equation}
Moreover, by the same letter we denote the corresponding $(0,3)$-tensor field of $Q$ with respect to $g$, \ie
\begin{equation}\label{2.2}
Q(x,y,z)=g\left(Q(x,y),z\right).
\end{equation}

\begin{prop}\label{thm-Q}
An affine connection $D$ is a natural connection on the Riemannian $\Pi$-manifold if and only if the following properties hold:
\begin{gather}
 Q(x,y,\ff z)-Q(x,\ff y,z)=F(x,y,z),\label{1a}
 \\%
 Q(x,y,z)=-Q(x,z,y).\label{1b}
\end{gather}
\end{prop}
\begin{proof}
Using \eqref{1} and \eqref{2.2}, we obtain the following relations
\begin{gather}
g(D_x \ff y,z)=g(\n_x \ff y,z)+ Q(x,\ff y,z),  \nonumber
\\[6pt]%
g(D_x y,\ff z)=g(\n_x y,\ff z)+ Q(x,y,\ff z).\nonumber
\end{gather}
We form the difference of the last two equalities and directly obtain the identity:
$$
\begin{array}{l}
g\bigl(\left(D_x \ff \right) y, z\bigr)=F(x, y, z)+ Q(x,\ff y,z)-Q(x,y,\ff z).
\end{array}
$$
Then the condition $D\ff=0$ is equivalent to \eqref{1a}.

We get sequentially
$$
\begin{array}{l}
\left(D_x g \right) (y, z)=g(\n_x y, z)+g(y,\n_x z)- g(D_x y,z)- g(y,D_x z)\\[6pt]
\phantom{\left(D_x g \right) (y, z)}=-Q(x,y,z)-Q(x,z,y).
\end{array}
$$
Therefore, the condition $Dg=0$ holds if and only if \eqref{1b} holds.

From \eqref{1} we obtain
\begin{equation}\label{Dxi=0}
\begin{array}{l}
g(D_x\xi,z)=g(\n_x\xi,z)+g(Q(x,\xi),z)=g(\n_x\xi,z)+Q(x,\xi,z).
\end{array}
\end{equation}
After that, from \lemref{lem-F} and \eqref{F-prop} we derive the following relation
\[
g(\n_x\xi,z)\allowbreak=F(x,\xi,\ff z).
\]
Substituting the latter result into \eqref{Dxi=0}, we get
\begin{equation*}
\begin{array}{l}
g(D_x\xi,z)=F(x,\xi,\ff z)+Q(x,\xi,z),
\end{array}
\end{equation*}
\ie the condition $D\xi=0$ is equivalent to the following relation
$$
F(x,\xi,\ff z)+Q(x,\xi,z)=0,
$$
which is a consequence of \eqref{1a}.

Since the relation $\eta(\cdot)=g(\cdot,\xi)$ holds, then, using $Dg=0$, we obtain that $D\xi=0$ is valid if and only if $D\eta=0$.
\end{proof}

\begin{thm}\label{thm-nat}
An affine connection $D$ is natural on a Riemannian $\Pi$-manifold if and only if
$$
D\ff=Dg=0.
$$
\end{thm}
\begin{proof}
In the proof of the preceding statement, we showed that the condition $D\ff=0$ is equivalent to \eqref{1a} and $Dg=0$ holds if and only if \eqref{1b} holds. In this way, according to \propref{thm-Q}, we complete the proof.
\end{proof}

\section{First natural connection on Riemannian $\Pi$-Manifolds}\label{sect-3}

Let $\Df$ denote an affine connection on $\M$ defined by
\begin{equation}\label{D1}
\begin{array}{l}
\Df_xy=\n_xy-\dfrac{1}{2}\bigl\{\left(\n_x\ff\right)\ff
y-\left(\n_x\eta\right)y\cdot\xi\bigr\}-\eta(y)\n_x\xi.
\end{array}
\end{equation}
Therefore, the potential $\Qf$ of $\Df$ with respect to $\n$ is defined by:
\begin{equation}\label{Q1}
\begin{array}{l}
\Qf(x,y)=-\dfrac{1}{2}\bigl\{\left(\n_x\ff\right)\ff
y-\left(\n_x\eta\right)y\cdot\xi\bigr\}-\eta(y)\n_x\xi.
\end{array}
\end{equation}

Using \eqref{strM}, \eqref{F} and \eqref{F-prop}, we verify that $\Df \ff=\Df g=0$.
Therefore, according to \thmref{thm-nat}, $\Df$ is a natural connection.

\begin{defn}\label{defn-D1}
The natural connection $\Df$, defined by \eqref{D1}, is called {first natural connection} on a Riemannian $\Pi$-manifold $\M$.
\end{defn}

Obviously, $\Df$ and $\n$ coincide only on a manifold of class $\F_0$. Therefore, $\n$ is a first natural connection when $\M \in \F_0$.

Let us remark that the restriction of $\Df$ on the paracontact distribution $\HH$ of $\M$ is another studied natural connection (called $P$-connection) on the corresponding Riemannian manifold equipped with an almost product structure (see e.g. \cite{Mek-P-con}).

\begin{thm}\label{thm-D1}
Let $\M$ be a $(2n+1)$-dimensional Riemannian $\Pi$-manifold belonging to the main classes $\F_i$ $(i=1,4,5,11)$. Then the first natural connection $\Df$ is determined by:
\begin{enumerate}
  \item If $\M \in \F_1$, then
  \[
  \begin{array}{l}
  \Df_xy=\n_xy-\dfrac{1}{4n}\left\{\ta(\ff y)\ff^2 x-\ta(\ff^2 y)\ff x+g(x,\ff y)\ff^2 \theta^\sharp - g(\ff x,\ff y)\ff \theta^\sharp \right\}, 
  \end{array}
  \]
  where $\ta(\cdot)=g(\theta^\sharp,\cdot)$;
  \item If $\M \in \F_4$, then
  \[\Df_xy=\n_xy-\dfrac{1}{2n}\ta(\xi)\left\{g(x,\ff y)\xi - \eta(y)\ff x\right\};\]
  \item If $\M \in \F_5$, then
  \[\Df_xy=\n_xy-\dfrac{1}{2n}\ta^*(\xi)\left\{g(\ff x,\ff y)\xi - \eta(y)\ff^2 x\right\};\]
  \item If $\M \in \F_{11}$, then
  \[
  \begin{array}{l}
  \Df_xy=\n_xy-\eta(x)\left\{\omega(\ff y)\xi - \eta(y)\ff \omega^\sharp\right\},
  \end{array}
  \]
  where $\omega(\cdot)=g(\omega^\sharp,\cdot)$.
\end{enumerate}
\end{thm}
\begin{proof}
We present the proof of the theorem in the first considered case, \ie $\M \in \F_1$.

The potential $\Qf$ has the following form given in \eqref{Q1}:
\[
\Qf(x,y)=-\dfrac{1}{2}\bigl\{\left(\n_x\ff\right)\ff
y-\left(\n_x\eta\right)y\cdot\xi\bigr\}-\eta(y)\n_x\xi.
\]
Using \eqref{F}, \lemref{lem-F} and the analogous definitions of \eqref{1} and \eqref{2.2} for $\Qf$
\begin{equation}\label{1-D1}
\Df_xy=\n_xy+\Qf(x,y),
\end{equation}
\begin{equation}\label{2.2-D1}
\Qf(x,y,z)=g\left(\Qf(x,y),z\right),
\end{equation}
we obtain the corresponding form of $\Qf$ as a tensor of type $(0,3)$
\begin{equation*}\label{Q1-prop}
\begin{array}{l}
\Qf(x,y,z)=-\dfrac{1}{2}\bigl\{F(x,\ff y,z)+\eta(z)F(x,\ff y,\xi)\bigr\}+\eta(y)F(x,\ff z,\xi).
\end{array}
\end{equation*}

Applying the definition condition of $F$ in $\F_1$ from \eqref{Fcon} 
\begin{equation}\label{F1-prop}
\begin{array}{l}
F(x,y,z)=\dfrac{1}{2n}\bigl\{g(\ff x,\ff y)\ta (\ff ^{2}z) +g(\ff x,\ff z)\ta (\ff^{2}y)\\[15pt]
\phantom{F(x,y,z)=\dfrac{1}{2n}\bigl\{}
-g(x,\ff y)\ta (\ff z)-g(x,\ff z)\ta(\ff y)\bigr\}
\end{array}
\end{equation}
in the latter formula and using \eqref{strM} and \eqref{strM2}, we obtain
\begin{equation}\label{Q1-F1-xyz}
\begin{array}{l}
\Qf(x,y,z)=-\dfrac{1}{4n}\left\{
g(\ff x,\ff^2 y)\ta(\ff^2 z)-g(x,\ff^2 y)\ta(\ff z) \right.\\[6pt]
\phantom{\Qf(x,y,z)=-\dfrac{1}{4n}\left\{\right.}\left.
g(\ff x,\ff z)\ta(\ff y)-g(x,\ff z)\ta(\ff^2 y) \right\}.
\end{array}
\end{equation}
From the latter equality and \eqref{2.2-D1} we get
\begin{equation}\label{Q1-F1-xy}
\begin{array}{l}
\Qf(x,y)=-\dfrac{1}{4n}\left\{\ta(\ff y)\ff^2 x-\ta(\ff^2 y)\ff x +g(x,\ff y)\ff^2 \theta^\sharp - g(\ff x,\ff y)\ff \theta^\sharp \right\},
\end{array}
\end{equation}
where $\ta(\cdot)=g(\theta^\sharp,\cdot)$.

Thus, we establish the truthfulness of the first statement in the theorem, con\-si\-de\-ring \eqref{1-D1}.
The other cases are proved in a similar way.
\end{proof}

Let $\Tf$ denote the torsion tensor of $\Df$, i.\hspace{.5pt}e., according to \eqref{T-def}, we have
$$
\Tf(x,y)=\Df_xy-\Df_yx-[x,y].
$$
Then, using \eqref{D1}, we obtain
\begin{equation}\label{D1-T}
\begin{array}{l}
\Tf(x,y)=-\dfrac{1}{2}\left\{(\n_x\ff)\ff y-(\n_y\ff)\ff x -\D \eta (x,y)\xi\right\}+ \eta(x)\n_y \xi - \eta(y)\n_x \xi.
\end{array}
\end{equation}
Let us remark that $\Df$ is not a symmetric connection since obviously $\Tf$ is nonzero.

The corresponding $(0,3)$-tensor with respect to $g$ is determined as follows
\begin{equation}\label{D1-T-03}
\begin{array}{l}
\Tf(x,y,z)=g(\Tf(x,y),z).
\end{array}
\end{equation}
Then, by \eqref{D1-T}, \eqref{F} and \lemref{lem-F}, we get
\begin{equation}\label{D1-Txyz}
\begin{array}{l}
\Tf(x,y,z)=-\dfrac{1}{2}\left\{F(x,\ff y,z)-F(y,\ff x,z)\right\}\\[6pt]
\phantom{\Tf(x,y,z)=}-\dfrac{1}{2}\eta(z)\left\{F(x,\ff y,\xi)-F(y,\ff x,\xi)\right\}\\[6pt]
\phantom{\Tf(x,y,z)=}+\eta(y)F(x,\ff z,\xi)-\eta(x)F(y,\ff z,\xi).
\end{array}
\end{equation}

We apply \eqref{F=NN} in \eqref{D1-Txyz}. Thus, taking into account \eqref{NN-prop}, we obtain the form of the torsion of the first natural connection with respect to $N$ and $\N$:
\begin{equation}\label{T=NN}
\begin{array}{l}
\Tf(x,y,z)=-\dfrac{1}{8}\bigl\{2N(\ff x,\ff y,z)+N(\ff x,z,\ff y)-N(\ff y,z,\ff x)\\[6pt]
\phantom{\Tf(x,y,z)=-\dfrac{1}{8}\bigl\{}
+\N(\ff x,z,\ff y)-\N(\ff y,z,\ff x)\bigr\}\\[6pt]
\phantom{\Tf(x,y,z)=}
+\dfrac{1}{4}\eta(x)\bigl\{2N(\xi,\ff y,\ff z)-N(\ff y,\ff z,\xi)\\[6pt]
\phantom{\Tf(x,y,z)=+\dfrac{1}{4}\eta(x)\bigl\{}
+2\eta(z)\N(\xi,\xi,\ff^2 y)-\N(\ff y,\ff z,\xi)\bigr\}\\[6pt]
\phantom{\Tf(x,y,z)=}
-\dfrac{1}{4}\eta(y)\bigl\{2N(\xi,\ff x,\ff z)-N(\ff x,\ff z,\xi)\\[6pt]
\phantom{\Tf(x,y,z)=-\dfrac{1}{4}\eta(y)\bigl\{}
+2\eta(z)\N(\xi,\xi,\ff^2 x)-\N(\ff x,\ff z,\xi)\bigr\}\\[6pt]
\phantom{\Tf(x,y,z)=}
-\dfrac{1}{8}\eta(z)\bigl\{2N(\ff x,\ff y,\xi)+N(\ff x,\xi,\ff y)-N(\ff y,\xi,\ff x)\\[6pt]
\phantom{\Tf(x,y,z)=-\dfrac{1}{4}\eta(y)\bigl\{}
+\N(\ff x,\xi,\ff y)-\N(\ff y,\xi,\ff x)\bigr\}.
\end{array}
\end{equation}

We use \eqref{T=NN} and the decomposition in \eqref{HHVV}, \eqref{hv} and \eqref{Xhv} to obtain the following form of $\Tf$ regarding the pair $N$ and $\N$ with respect to the horizontal and the vertical components of the vector fields:
\begin{subequations}\label{T1Nhv}
\begin{equation*}
\begin{split}
\Tf(x,y,z) =-\dfrac{1}{8}\bigl\{
&\sx N(x^{\mathrm{h}},y^{\mathrm{h}},z^{\mathrm{h}}) + N(x^{\mathrm{h}},y^{\mathrm{h}},z^{\mathrm{h}}) \\[6pt]
&+ \N(y^{\mathrm{h}},z^{\mathrm{h}},x^{\mathrm{h}})-\N(z^{\mathrm{h}},x^{\mathrm{h}},y^{\mathrm{h}})\bigr\} \\[6pt]
-\dfrac{1}{4}\bigl\{
&2N(x^{\mathrm{h}},y^{\mathrm{h}},z^{\mathrm{v}})+N(y^{\mathrm{h}},z^{\mathrm{v}},x^{\mathrm{h}})
+N(z^{\mathrm{v}},x^{\mathrm{h}},y^{\mathrm{h}}) \\[6pt]
\end{split}
\end{equation*}
\begin{equation*}
\begin{split}
&+2N(x^{\mathrm{v}},y^{\mathrm{h}},z^{\mathrm{h}}) + N(y^{\mathrm{h}},z^{\mathrm{h}},x^{\mathrm{v}})
+2N(x^{\mathrm{h}},y^{\mathrm{v}},z^{\mathrm{h}})  \\[6pt]
&+ N(z^{\mathrm{h}},x^{\mathrm{h}},y^{\mathrm{v}})+2\N(y^{\mathrm{h}},z^{\mathrm{h}},x^{\mathrm{v}})
-\N(z^{\mathrm{v}},x^{\mathrm{h}},y^{\mathrm{h}}) \\[6pt]
&-\N(z^{\mathrm{h}},x^{\mathrm{h}},y^{\mathrm{v}})-2\N(z^{\mathrm{v}},x^{\mathrm{v}},y^{\mathrm{h}})
+ 2\N(y^{\mathrm{v}},z^{\mathrm{v}},x^{\mathrm{h}})\bigr\},
\end{split}
\end{equation*}
\end{subequations}
where $\sx$ stands for the cyclic sum by the three arguments.

\begin{thm}\label{thm-T1}
Let $\M$ be a $(2n+1)$-dimensional Riemannian $\Pi$-manifold belonging to the main classes $\F_i$ $(i=1,4,5,11)$. Then the torsion tensor $\Tf$ of the first natural connection $\Df$ has the form:
\begin{enumerate}
  \item If $\M \in \F_1$, then
  \[
  \begin{array}{l}
  \Tf(x,y)=-\dfrac{1}{4n}\left\{\ta(\ff y)\ff^2 x-\ta(\ff x)\ff^2 y
    +\ta(\ff^2 x)\ff y-\ta(\ff^2 y)\ff x\right\};
  \end{array}
  \]
  \item If $\M \in \F_4$, then
  \[\Tf(x,y)=\dfrac{1}{2n}\ta(\xi)\left\{\eta(y)\ff x-\eta(x)\ff y\right\};\]
  \item If $\M \in \F_5$, then
  \[\Tf(x,y)=\dfrac{1}{2n}\ta^*(\xi)\left\{\eta(y)\ff^2 x-\eta(x)\ff^2 y\right\};\]
  \item If $\M \in \F_{11}$, then
  \[\Tf(x,y)=\left\{\eta(y)\omega(\ff x) - \eta(x)\omega(\ff y)\right\}\xi.\]
\end{enumerate}
\end{thm}
\begin{proof}
We present the proof of the theorem in the first considered case, \ie $\M \in \F_1$.

We apply \eqref{F1-prop} in \eqref{D1-Txyz} and taking into account \eqref{strM} and \eqref{strM2}, we obtain
\begin{equation*}\label{T1-F1-xyz}
\begin{array}{l}
\Tf(x,y,z)=-\dfrac{1}{4n}\left\{
g(\ff x,\ff z)\ta(\ff y) - g(\ff y,\ff z)\ta(\ff x)\right.\\[6pt]
\phantom{\Tf(x,y,z)=-\dfrac{1}{4n}\left\{\right.}\left.
-g(x,\ff z)\ta(\ff^2 y) + g(y,\ff z)\ta(\ff^2 x)\right\}.
\end{array}
\end{equation*}
The form of $\Tf$ in case 1 follows from the last expression and \eqref{D1-T-03}.

Thus, we establish the truthfulness of the first statement in the theorem.
The other cases are proved in a similar way.
\end{proof}

Similarly to \eqref{t}, we define torsion forms $\dot{t}$, $\dot{t}^*$ and $\widehat{\dot{t}}$ for $\Tf$ with respect to a basis $\left\{\xi;e_i\right\}$ of $T_p\mathcal{M}$ $(i=1,2,\dots,2n; p\in \mathcal{M})$:
\begin{equation}\label{t1}
\begin{array}{l}
\dot{t}(x)=g^{ij}\Tf(x,e_i,e_j),\qquad \dot{t}^*(x)=g^{ij}\Tf(x,e_i,\ff e_j),\qquad
\widehat{\dot{t}}(x)=\Tf(x,\xi,\xi).
\end{array}
\end{equation}
Using \eqref{D1-Txyz}, \eqref{t1} and $\eta(e_i)=0$ $(i=1,\dots,2n)$, we obtain
\[
\begin{split}
\dot{t}(x)=-\dfrac{1}{2}g^{ij}\bigl\{&F(x,\ff_i^m e_m,e_j)-F(e_i,\ff x,e_j)+2\eta(x)F(e_i,\ff_j^m e_m,\xi)\bigr\}.
\end{split}
\]
On the one hand, by \eqref{strM} and the identities $\ff_i^k\ff_j^sg^{ij}=g^{ks}-\xi^k\xi^s$ and $\eta(e_i)=0$ $(i=1,\dots,2n)$, for the first addend of the last equality we get
\[
\begin{split}
g^{ij}F(x,\ff_i^s e_s,e_j)&=g^{ij}F(x,\ff_i^s e_s,\ff_j^m\ff_m^l
e_l)=\ff_i^s\ff_j^lg^{ij}F(x,e_s,\ff_l^m
e_m)\\[6pt]
&=g^{sl}F(x,e_s,\ff_l^m e_m)-\xi^s \xi^lF(x,e_s,\ff_l^m
e_m)=g^{ij}F(x,e_i,\ff_j^l e_l).
\end{split}
\]
On the other hand, from \eqref{F-prop}, we have for it
\[
\begin{split}
g^{ij}F(x,\ff_i^s e_s,e_j)&=-g^{ij}F(x,\ff_i^m\ff_m^s e_s,\ff_j^l
e_l)=-g^{ij}F(x,e_i,\ff_j^le_l).
\end{split}
\]
Therefore $g^{ij}F(x,\ff_i^s e_s,e_j)=g^{ij}F(x,e_i,\ff_j^le_l)=0$.

Thus, according to \eqref{t}, we obtain the following formula
\begin{equation}\label{t-fB}
\dot{t}(x)=\dfrac{1}{2}\ta(\ff x)-\ta^*(\xi)\eta(x).
\end{equation}
By an analogous approach, we calculate the form of $\dot{t}^*$ and $\widehat{\dot{t}}$ as follows:
\begin{equation}\label{t*-om-fB}
\begin{array}{l}
\dot{t}^*(x)=\dfrac{1}{2}\ta^*(\ff x)-\ta(\xi)\eta(x),\\[6pt]
\widehat{\dot{t}}(x)=\om(\ff x).
\end{array}
\end{equation}

Taking into account \eqref{ta-prop}, \eqref{t-fB} and \eqref{t*-om-fB}, we obtain the following relations between the torsion forms $\dot{t}$, $\dot{t}^*$ and the Lee forms $\ta$, $\ta^*$:
\begin{equation}\label{t1t1*}
\begin{array}{c}
\dot{t}^*\circ\ff=\dot{t}\circ\ff^2,\\[6pt]
2\dot{t}\circ\ff=\ta\circ\ff^2,\qquad
2\dot{t}\circ\ff^2=\ta\circ\ff,\\[6pt]
2\dot{t}^*\circ\ff=\ta^*\circ\ff^2,\qquad
2\dot{t}^*\circ\ff^2=\ta^*\circ\ff.
\end{array}
\end{equation}

\begin{cor}\label{thm:FiT1}
Let $\M$ be a $(2n+1)$-dimensional Riemannian $\Pi$-manifold belonging to the main classes $\F_i$ $(i=1,4,5,11)$. Then the torsion tensor $\Tf$ of the first natural connection $\Df$ is expressed by its torsion forms $\dot{t}$, $\dot{t}^*$ and $\widehat{\dot{t}}$ as follows:
\[
\begin{array}{rl}
\F_1:\; &\Tf(x,y)=-\dfrac{1}{2n}\left\{\dot{t}(\ff^2 y)\ff^2 x-\dot{t}(\ff^2 x)\ff^2 y
+\dot{t}(\ff x)\ff y-\dot{t}(\ff y)\ff x\right\}; \\[6pt]
\F_4:\; &\Tf(x,y)=-\dfrac{1}{2n}\dot{t}^*(\xi)\left\{\eta(y)\ff x-\eta(x)\ff y\right\};\\[6pt]
\F_5:\; &\Tf(x,y)=-\dfrac{1}{2n}\dot{t}(\xi)\left\{\eta(y)\ff^2 x-\eta(x)\ff^2 y\right\};\\[6pt]
\F_{11}:\; &\Tf(x,y)=\left\{\eta(y)\widehat{\dot{t}}(x) - \eta(x)\widehat{\dot{t}}(y)\right\}\xi.
\end{array}
\]
\end{cor}
\begin{proof}
We obtain the expression of $\Tf$ using its form from \thmref{thm-T1} and the relations \eqref{t1t1*} between the torsion forms and the Lee forms.
\end{proof}

Let $\Rf$ denote the curvature tensor for the first natural connection $\Df$. Similarly to the definitions \eqref{R_xyz} and \eqref{R_xyzw} of $R$ regarding $\n$, we define $\Rf$ as a tensor of type $(1,3)$ and $(0,4)$ for $\Df$, respectively, by:
\begin{equation}\label{R1-xyz}
\Rf(x,y)z=\Df_x\Df_yz-\Df_y\Df_xz-\Df_{[x,y]}z,
\end{equation}
\begin{equation}\label{R1-xyzw}
\Rf(x,y,z,w)=g\left(\Rf(x,y)z,w\right).
\end{equation}

\begin{thm}\label{thm-R1}
Let $\M$ be a $(2n+1)$-dimensional Riemannian $\Pi$-manifold belonging to the main classes $\F_i$ $(i=1,4,5,11)$. Then the curvature tensor $\Rf$ of the first natural connection $\Df$ has the form:
\begin{enumerate}
  \item If $\M \in \F_1$, then
  \[
  \begin{array}{l}
  \Rf(x,y,z,w)=R(x,y,z,w)\\[6pt]
  \phantom{\Rf(x,y,z,w)=}+\dfrac{1}{4n}\left\{\big(g^*\owedge S_1 - g^{**} \owedge S_2\big)(x,y,z,w) \right.\\[6pt]
  \end{array}
  \]
  \[
  \begin{array}{l}
  \phantom{\Rf(x,y,z,w)=+\dfrac{1}{4n}\left\{\right.}\left.- \ta(\ff \theta^\sharp)\big( g^* \owedge g^{**} \big)(x,y,z,w) \right.\\[6pt]
  \phantom{\Rf(x,y,z,w)=+\dfrac{1}{4n}\left\{\right.}\left.
  - \ta(\ff^2 \theta^\sharp)\big(g \owedge g^{**}+g^* \owedge \tilde{g}-\tilde{g}\owedge g\big)(x,y,z,w)
  \right\},
  \end{array}
  \]
  where
  \[
  \begin{array}{l}
  S_1(x,y)=\left(\n_x \left(\ta\circ\ff^2\right)\right)(y)+\dfrac{1}{4n}\left\{\ta(\ff x)\ta(\ff^2 y)+\ta(\ff^2 x)\ta(\ff y)\right\},\\[6pt]
  S_2(x,y)=\left(\n_x \left(\ta\circ\ff\right)\right)(y)+\dfrac{1}{4n}\left\{\ta(\ff^2 x)\ta(\ff^2 y)+\ta(\ff x)\ta(\ff y)\right\};
  \end{array}
  \]
  \item If $\M \in \F_4$, then
  \[
  \begin{array}{l}
  \Rf(x,y,z,w)=R(x,y,z,w)\\[6pt]
  \phantom{\Rf(x,y,z,w)=}+\dfrac{1}{2n}\left\{x\big(\ta(\xi)\big)\left\{(\eta\otimes\eta) \owedge g^*\right\}(\xi,y,z,w)\right.\\[6pt]
  \phantom{\Rf(x,y,z,w)==+\dfrac{1}{2n}\left\{\right.}
  \left.-y\big(\ta(\xi)\big)\left\{(\eta\otimes\eta) \owedge g^*\right\}(\xi,x,z,w)\right\}\\[6pt]
  \phantom{\Rf(x,y,z,w)=}-\dfrac{1}{8n^2}\left(\ta(\xi)\right)^2 \left\{ 2(\eta\otimes\eta) \owedge g - g^*\owedge g^*\right\}(x,y,z,w);
  \end{array}
  \]
  \item If $\M \in \F_5$, then
  \[
  \begin{array}{l}
  \Rf(x,y,z,w)=R(x,y,z,w)\\[6pt]
  \phantom{\Rf(x,y,z,w)=}+\dfrac{1}{4n}\left\{x\big(\ta^*(\xi)\big)\left\{g \owedge g\right\}(\xi,y,z,w)\right.\\[6pt]
  \phantom{\Rf(x,y,z,w)=+\dfrac{1}{2n}\left\{\right.}
  \left.-y\big(\ta^*(\xi)\big)\left\{g \owedge g\right\}(\xi,x,z,w)\right\}\\[6pt]
  \phantom{\Rf(x,y,z,w)=}
  +\dfrac{1}{8n^2}\left(\ta^*(\xi)\right)^2 \left\{g \owedge g\right\}(x,y,z,w);
  \end{array}
  \]
  \item If $\M \in \F_{11}$, then
  \[
  \begin{array}{l}
  \Rf(x,y,z,w)=R(x,y,z,w)\\[6pt]
  \phantom{\Rf(x,y,z,w)=}-\left\{(\eta\otimes\eta) \owedge S_3\right\}(x,y,z,w),
  \end{array}
  \]
  where
  \[
  S_3(x,y)=(\n_x \om) (\ff y)+\om(\ff x)\om(\ff y).
  \]
\end{enumerate}
\end{thm}
\begin{proof}
We present the proof of the theorem in the first considered case, \ie $\M \in \F_1$.

Using \eqref{R1-xyz} and \eqref{R1-xyzw} together with \eqref{1-D1}, \eqref{2.2-D1} and the analogous relation of \eqref{1b} for $\Qf$, we get the following form of $\Rf$ for an arbitrary Riemannian $\Pi$-manifold $\M$:
\begin{equation}\label{R1}
\begin{array}{l}
\Rf(x,y,z,w)=R(x,y,z,w)+(\n_x\Qf)(y,z,w)-(\n_y\Qf)(x,z,w)\\[6pt]
\phantom{\Rf(x,y,z,w)=R(x,y,z,w)}+g(\Qf(x,z),\Qf(y,w))-g(\Qf(y,z),\Qf(x,w)).
\end{array}
\end{equation}

Taking into account \eqref{Q1-F1-xyz}, \eqref{F} and \eqref{F-prop}, we obtain
\begin{equation}\label{R11-dok1}
\begin{array}{l}
\left(\n_x\Qf\right)(y,z,w)=-\dfrac{1}{4n}\left\{x\left(\ta(\ff^2w)\right)g(y,\ff z)+x\left(\ta(\ff z)\right)g(\ff y,\ff w)\right.\\[6pt]
\phantom{\left(\n_x\Qf\right)(y,z,w)=-\dfrac{1}{4n}\left\{\right.}\left.
-x\left(\ta(\ff w)\right)g(\ff y,\ff z)-x\left(\ta(\ff^2 z)\right)g(y,\ff w)\right.\\[6pt]
\phantom{\left(\n_x\Qf\right)(y,z,w)=-\dfrac{1}{4n}\left\{\right.}\left.
+\ta(\ff^2w)F(x,y,z)-\ta(\ff^2z)F(x,y,w)\right.\\[6pt]
\phantom{\left(\n_x\Qf\right)(y,z,w)=-\dfrac{1}{4n}\left\{\right.}\left.
-\ta(\ff w)\left\{F(x,y,\ff z)+F(x,z,\ff y)\right\}\right.\\[6pt]
\phantom{\left(\n_x\Qf\right)(y,z,w)=-\dfrac{1}{4n}\left\{\right.}\left.
+\ta(\ff z)\left\{F(x,y,\ff w)+F(x,w,\ff y)\right\}\right.\\[6pt]
\phantom{\left(\n_x\Qf\right)(y,z,w)=-\dfrac{1}{4n}\left\{\right.}\left.
-\ta(\ff^2 \n_xw)g(y,\ff z)+\ta(\ff \n_xw)g(\ff y,\ff z)\right.\\[6pt]
\phantom{\left(\n_x\Qf\right)(y,z,w)=-\dfrac{1}{4n}\left\{\right.}\left.
+\ta(\ff^2 \n_xz)g(y,\ff w)-\ta(\ff \n_xz)g(\ff y,\ff w)\right\}.
\end{array}
\end{equation}

Then, using \eqref{Q1-F1-xy}, we get:
\begin{equation}\label{R11-dok2}
\begin{array}{l}
g\left(\Qf(x,z),\Qf(y,w)\right)\\[6pt]
\phantom{g\left(\right)}=-\dfrac{1}{16n^2}\left\{\vphantom{\ta(\ff^2 \theta^\sharp)}
\ta(\ff z)\left[\ta(\ff w)g(\ff x,\ff y)-\ta(\ff^2 w)g(\ff x,y)\right.\right.\\[6pt]
\phantom{g\left(\right)=-\dfrac{1}{16n^2}\left\{\ta(\ff z)\left[\right.\right.}
\left.+\ta(\ff^2 x)g(y,\ff w)-\ta(\ff x)g(\ff y,\ff w)\right]\\[6pt]
\phantom{g\left(\right)=-\dfrac{1}{16n^2}\;}
\left.-\ta(\ff^2 z)\left[\ta(\ff w)g(x,\ff y)-\ta(\ff^2 w)g(\ff x,\ff y)\right.\right.\\[6pt]
\phantom{g\left(\right)=-\dfrac{1}{16n^2}\left\{-\ta(\ff^2 z)\left[\right.\right.}
+\ta(\ff x)g(y,\ff w)\left.-\ta(\ff^2 x)g(\ff y,\ff w)\right]\\[6pt]
\phantom{g\left(\right)=-\dfrac{1}{16n^2}\;}
\left.+g(x,\ff z)\left[\ta(\ff^2 y)\ta(\ff w)-\ta(\ff y)\ta(\ff^2 w)\right.\right.\\[6pt]
\phantom{g\left(\right)=-\dfrac{1}{16n^2}\;+g(x,\ff z)\;}
\left.+\ta(\ff^2 \theta^\sharp)g(y,\ff w)-\ta(\ff \theta^\sharp)g(\ff y,\ff w)\right]\\[6pt]
\phantom{g\left(\right)=-\dfrac{1}{16n^2}\;}
\left.-g(\ff x,\ff z)\left[\ta(\ff y)\ta(\ff w)-\ta(\ff^2 y)\ta(\ff^2 w)\right.\right.\\[6pt]
\phantom{g\left(\right)=-\dfrac{1}{16n^2}\;-g(\ff x,\ff z)\;}\left.\left.
+\ta(\ff \theta^\sharp)g(y,\ff w)-\ta(\ff^2 \theta^\sharp)g(\ff y,\ff w)\right]\right\}.
\end{array}
\end{equation}

Applying \eqref{R11-dok1} and \eqref{R11-dok2} into \eqref{R1} and using \eqref{strM} and \eqref{strM2} as well as the notations \eqref{g***} and \eqref{KulNom}, we obtain the form of $\Rf$ presented in the theorem.

Thus, we establish the truthfulness of the first statement in the theorem.
The other cases are proved in a similar way.
\end{proof}

Similarly to the definitions \eqref{ro}, \eqref{tau} \eqref{ro*-tau*} for $\rho$, $\tau$, $\rho^*$ and $\tau^*$ regarding $R$, we define the corresponding ones with respect to $\Rf$ as follows:
\begin{equation*}\label{ro-tau-1}
\begin{array}{ll}
\dot{\rho}(x,y)=g^{ij}\Rf(e_{i},x,y,e_{j}),\qquad
&\dot{\tau}=g^{ij}\dot{\rho}(e_{i},e_{j}),\\[6pt]
\dot{\rho}^*(x,y)=g^{ij}\Rf(e_{i},x,y,\ff e_{j}),\qquad
&\dot{\tau}^*=g^{ij}\dot{\rho}^*(e_{i},e_{j}).
\end{array}
\end{equation*}

\begin{cor}\label{cor-R1}
Let $\M$ be a $(2n+1)$-dimensional Riemannian $\Pi$-manifold belonging to the main classes $\F_i$ $(i=1,4,5,11)$. Then the following relations for the Ricci tensors and the scalar curvatures with respect to $\Df$ and $\n$ hold:
\begin{enumerate}
  \item If $\M \in \F_1$, then
  \[
  \begin{array}{l}
  \dot{\rho}(y,z)=\rho(y,z)\\[6pt]
  \phantom{\dot{\rho}(y,z)=}+\dfrac{1}{2}\left\{\left(\n_y \left(\ta\circ\ff\right)\right)(z)+\dfrac{1}{4n}\left\{\ta(\ff^2 y)\ta(\ff^2 z)
  +\ta(\ff y)\ta(\ff z)\right\}\right\} \\[6pt]
  \phantom{\dot{\rho}(y,z)=}-\dfrac{1}{4n}\left\{\left( \Div(\ta\circ\ff^2)-\dfrac{4n^2-4n-1}{2n}\,\ta(\ff \theta^\sharp)\right.\right.\\[6pt]
  \phantom{\dot{\rho}(y,z)=-\dfrac{1}{4n}\left\{\right.\left(\right.\Div(\ta\circ\ff^2)-aaa\;}\left.
  +2(n-1)\,\ta(\ff^2 \theta^\sharp) \vphantom{\dfrac{a^1}{b}}\right) g(y,\ff z)\\[6pt]
  \phantom{\dot{\rho}(y,z)=-\dfrac{1}{4n}\left\{\right.}\left.
  -\left(\Div(\ta\circ\ff)+\dfrac{8n^2-8n+1}{2n}\,\ta(\ff^2 \theta^\sharp)\right) g(\ff y,\ff z)\right\},\\[6pt]
  \dot{\rho}^*(y,z)=\rho^*(y,z)\\[6pt]
  \phantom{\dot{\rho}^*(y,z)=}-\dfrac{1}{2}\left\{\left(\n_y \left(\ta\circ\ff^2\right)\right)(z)+\dfrac{1}{4n}\left\{\ta(\ff y)\ta(\ff^2 z)+\ta(\ff^2 y)\ta(\ff z)\right\}\right\} \\[6pt]
  \phantom{\dot{\rho}^*(y,z)=}+\dfrac{1}{4n}\left\{\left( \Div^*(\ta\circ\ff)+\dfrac{(2n-1)^2}{2n}\,\ta(\ff \theta^\sharp)\right.\right.\\[6pt]
  \phantom{\dot{\rho}^*(y,z)=+\dfrac{1}{4n}\left\{\right.\left(\right.\Div^*(\ta\circ\ff)aaaaaaa\;}\left.
  -2(n-1)\,\ta(\ff^2 \theta^\sharp) \vphantom{\dfrac{a^1}{b}}\right) g(\ff y,\ff z)\\[6pt]
  \phantom{\dot{\rho}^*(y,z)=+\dfrac{1}{4n}\left\{\right.}\left.
  -\left(\Div^*(\ta\circ\ff^2)-\dfrac{8n^2-8n-1}{2n}\,\ta(\ff^2 \theta^\sharp)\right) g(y,\ff z)\right\},\\[6pt]
  \dot{\tau}=\tau+\Div (\ta\circ\ff) + \dfrac{(2n-1)^2}{2n}\,\ta(\ff^2 \theta^\sharp),\\[6pt]
  \dot{\tau}^*=\tau^*+(n-1)\,\ta(\ff \theta^\sharp) -\dfrac{2n-3}{2}\,\ta(\ff^2 \theta^\sharp),
  \end{array}
  \]
  where $\Div (\ta)=g^{ij}\left(\n_{e_i}\ta\right)(e_j)$, $\Div^* (\ta)=g^{ij}\left(\n_{e_i}\ta\right)(\ff e_j)$;
  \item If $\M \in \F_4$, then
  \[
  \begin{array}{l}
  \dot{\rho}(y,z)=\rho(y,z)\\[6pt]
  \phantom{\dot{\rho}(y,z)=}-\dfrac{1}{2n}\left\{\xi\big(\ta(\xi)\big)g(y,\ff z) - \ff y\big(\ta(\xi)\big)\eta(z)\right\}\\[6pt]
  \phantom{\dot{\rho}(y,z)=}+\dfrac{1}{2n^2}\,\left(\ta(\xi)\right)^2\left\{g(y,z)+(n-1)\eta(y)\eta(z)\right\},\\[6pt]
  \dot{\rho}^*(y,z)=\rho^*(y,z)\\[6pt]
  \phantom{\dot{\rho}^*(y,z)=}+\dfrac{1}{2n}\left\{\ff^2 y\big(\ta(\xi)\big)-2n\,y\big(\ta(\xi)\big)\right\}\eta(z)\\[6pt]
  \phantom{\dot{\rho}^*(y,z)=}-\dfrac{2n-1}{4n^2}\,\left(\ta(\xi)\right)^2g(y,\ff z),\\[6pt]
  \dot{\tau}=\tau+\dfrac{1}{2n}\,\left(\ta(\xi)\right)^2,\\[6pt]
  \dot{\tau}^*=\tau^*-\xi\big(\ta(\xi)\big);
  \end{array}
  \]
\item If $\M \in \F_5$, then
  \[
  \begin{array}{l}
  \dot{\rho}(y,z)=\rho(y,z)\\[6pt]
  \phantom{\dot{\rho}(y,z)=}-\dfrac{1}{2n}\left\{\xi\big(\ta^*(\xi)\big)g(y,z) + (2n-1)\,y\big(\ta^*(\xi)\big)\eta(z)\right\}\\[6pt]
\phantom{\dot{\rho}(y,z)=}-\dfrac{1}{2n}\,\left(\ta^*(\xi)\right)^2 g(y,z),\\[6pt]
  \end{array}
  \]
  \[
  \begin{array}{l}
  \dot{\rho}^*(y,z)=\rho^*(y,z)\\[6pt]
  \phantom{\dot{\rho}^*(y,z)=}-\dfrac{1}{2n}\left\{\ff y\big(\ta^*(\xi)\big)\eta(z)\right\}+\dfrac{1}{4n^2}\,\left(\ta^*(\xi)\right)^2 g(y,\ff z),\\[6pt]
  \end{array}
  \]
  \[
  \begin{array}{l}
  \dot{\tau}=\tau-2\xi\big(\ta^*(\xi)\big)-\dfrac{2n+1}{2n}\,\big(\ta^*(\xi)\big)^2,\\[6pt]
  \dot{\tau}^*=\tau^*;
  \end{array}
  \]
  \item If $\M \in \F_{11}$, then
  \[
  \begin{array}{l}
  \dot{\rho}(y,z)=\rho(y,z)\\[6pt]
  \phantom{\dot{\rho}(y,z)=}
  + (\n_y \om) (\ff z)+\om(\ff y)\om(\ff z)\\[6pt]
  \phantom{\dot{\rho}(y,z)=}
  +\left\{\Div^*(\om)+\om(\ff^2 \omega^\sharp)\right\}\eta(y)\eta(z),\\[6pt]
  \dot{\rho}^*(y,z)=\rho^*(y,z)\\[6pt]
  \phantom{\dot{\rho}^*(y,z)=}
  +\left\{\Div(\om)+\om(\ff \omega^\sharp)\right\}\eta(y)\eta(z),\\[6pt]
  \dot{\tau}=\tau+2\left\{\Div^*(\om)+\om(\ff^2 \omega^\sharp)\right\},\\[6pt]
  \dot{\tau}^*=\tau^*+\left\{\Div(\om)+\om(\ff \omega^\sharp)\right\},
  \end{array}
  \]
  where $\Div (\om)=g^{ij}\left(\n_{e_i}\om\right)(e_j)$, $\Div^* (\om)=g^{ij}\left(\n_{e_i}\om\right)(\ff e_j)$;
\end{enumerate}
\end{cor}
\begin{proof}
We present the proof of the theorem in the first considered case, \ie $\M \in \F_1$.

Using \eqref{strM} and \eqref{strM2}, we easily compute $\dot{\rho}$ as the trace of $\Rf(x,y,z,w)$, given in \thmref{thm-R1} (1), by $g^{ij}$ for $x=e_i$ and $w=e_j$.

Similarly, we calculate the trace of $\Rf(x,y,z,w)$ by $g^{ij}$ for $x=e_i$ and $w=\ff e_j$ and we obtain the form of $\dot{\rho}^*$, again taking into account \eqref{strM} and \eqref{strM2}.

Finally, the values of $\dot{\tau}$ and $\dot{\tau}^*$ are obtained by calculating the traces of $\dot{\rho}(y,z)$ and $\dot{\rho}^*(y,z)$ by $g^{ij}$ for $y=e_i$ and $z=e_j$.

Thus, we establish the truthfulness of the first statement in the corollary.
The other cases are proved in a similar way.
\end{proof}

\section{Example}\label{sect-4}

In this section, we consider a known example of a Riemannian $\Pi$-manifold of dimension 5, recalling some obtained results for it and presenting new ones related to the studied theory.

The authors of \cite{ManVes18} studied the so-called {paracontact almost paracomplex Rie\-man\-nian manifolds}, which are Riemannian $\Pi$-manifolds having the property $2g(x,\ff y)=(\nabla _x\eta)( y)+(\nabla _y\eta)(x)$.

According to the classification of the considered manifolds from \cite{ManSta01}, we denote by ${\F_4}'$ a subclass of $\F_4$ which is defined by the condition $\ta(\xi)=-2n$.
It is important to note that ${\F_4}'$ and ${\F_0}$ are subclasses of ${\F_4}$ but without common elements.

A paracontact almost paracomplex Riemannian manifold having the additional condition $\ff x=\n_x \xi$
is called a {para-Sasakian paracomplex Riemannian manifold} and it belongs to the class ${\F_4}'$ \cite{ManVes18}.

In \cite{IvMan2}, the same class of manifolds is obtained by a cone construction of a paraholomorphic paracomplex Riemannian manifold. There they are called para-Sasaki-like paracontact paracomplex Riemannian manifolds.

Let us consider a Lie group $\mathcal{G}$ of dimension $5$ (\ie $n=2$) which has a basis of left-invariant vector fields $\{e_0,\dots, e_{4}\}$ and the corresponding Lie algebra is defined for $\lm,\, \mu\in\R$ by the following commutators
\begin{equation}\label{comEx1}
\begin{array}{ll}
[e_0,e_1] = \lm e_2 - e_3 + \mu e_4,\qquad
&[e_0,e_2] = - \lm e_1 - \mu e_3 - e_4,\\[6pt]
[e_0,e_3] = - e_1  + \mu e_2 + \lm e_4,\qquad
&[e_0,e_4] = -\mu e_1 - e_2 - \lm e_3.
\end{array}
\end{equation}
The defined Lie group $\mathcal{G}$ is equipped with an invariant Riemannian $\Pi$-structure $(\phi, \xi, \eta, g)$ as follows:
\begin{equation}\label{strEx1}
\begin{array}{l}
\xi=e_0, \quad \ff  e_1=e_{3},\quad  \ff e_2=e_{4},\quad \ff  e_3=e_{1},\quad \ff  e_4=e_{2}, \\[6pt]
\eta(e_1)=\eta(e_2)=\eta(e_3)=\eta(e_4)=0,\quad \eta(e_0)=1, \\[6pt]
g(e_0,e_0)=g(e_1,e_1)=g(e_2,e_2)=g(e_{3},e_{3})=g(e_{4},e_{4})=1, \\[6pt]
g(e_i,e_j)=0,\quad i,j\in\{0,1,\dots,4\},\; i\neq j.
\end{array}
\end{equation}
It is proved that the constructed manifold $(\mathcal{G}, \phi, \xi, \eta, g)$ is a para-Sassaki-like paracontact paracomplex Riemannian manifold, \ie $(\mathcal{G}, \phi, \xi, \eta, g) \in \F_4$ \cite{IvMan2}.

Using \eqref{comEx1}, \eqref{strEx1} and the well-known Koszul equality regarding $g$ and $\n$, we calculate the components of the Levi-Civita connection and the nonzero ones of them are the following:
\begin{equation}\label{n-ex1}
\begin{array}{c}
\begin{array}{ll}
\n_{e_0} e_1 = \lm e_2+\mu e_4,\quad & \n_{e_1}e_0 = e_3,\\[6pt]
\n_{e_0} e_2 = -\lm e_1-\mu e_3, \quad & \n_{e_2} e_0 = e_4,\\[6pt]
\n_{e_0} e_3 = \mu e_2 + \lm e_4,\quad & \n_{e_3} e_0 = e_1, \\[6pt]
\n_{e_0} e_4 = -\mu e_1 - \lm e_3, \quad & \n_{e_4}e_0 = e_2,
\end{array}
\\[6pt]
\begin{array}{c}\\[-6pt]
\n_{e_1}e_3 = \n_{e_2} e_4 = \n_{e_3} e_1 = \n_{e_4}e_2 = - e_0.
\end{array}
\end{array}
\end{equation}

Taking into account \eqref{comEx1}, \eqref{strEx1} and \eqref{n-ex1}, we calculate the components $R_{ijkl}=R(e_i,e_j,e_k,e_l)$, $\rho_{ij} =\rho(e_i,e_j)$ and $\rho^*_{ij}=\rho^*(e_i,e_j)$ as well as the values of $\tau$ and $\tau^*$. The nonzero ones of them are determined by the following equalities and their well-known symmetries and antisymmetries:
\begin{equation}\label{R-ex1}
\begin{array}{l}
R_{0101}=R_{0202}=R_{0303}=R_{0404}=R_{1331}=R_{2442}=R_{1234}=R_{1432}=1,\\[6pt]
\rho_{00}=-4, \qquad \rho^*_{13}=\rho^*_{24}=-3, \qquad \tau=-4.
\end{array}
\end{equation}

Let us consider the first natural connection $\Df$ on $(\mathcal{G},\allowbreak{}\ff,\xi,\eta,g)$ defined by \eqref{D1}. Then, by the relation between $\Df$ and $\n$ in the case of $\F_4$ from \thmref{thm-D1}, and using \eqref{n-ex1}, we obtain the components of $\Df$. The nonzero ones of them are the following:
\begin{equation}\label{Df-ex1}
\begin{array}{l}
\Df_{e_0}e_1=\lm e_2+\mu e_4, \qquad
\Df_{e_0}e_2=-\lm e_1-\mu e_3, \\[6pt]
\Df_{e_0}e_3=\mu e_2+\lm e_4, \qquad
\Df_{e_0}e_4=-\mu e_1-\lm e_3.
\end{array}
\end{equation}

\begin{prop}\label{ex1-plosko}
The Riemannian $\Pi$-manifold $\GG$ has a flat first natural connection $\Df$, \ie $\Rf=0$.
\end{prop}
\begin{proof}
Using \eqref{Df-ex1} and \eqref{R1-xyz}, we establish that the components of $\Rf$ vanish. Thus we prove the assertion.
\end{proof}

\begin{cor}\label{ex1-cor}
The Riemannian $\Pi$-manifold $\GG$ is Ricci-flat and scalar-flat with respect to the first natural connection $\Df$, \ie $\dot{\rho}=0$ and $\dot{\tau}=0$.
\end{cor}
\begin{proof}
The truthfulness of the corollary is obvious bearing in mind \propref{ex1-plosko}.
\end{proof}

Taking into account \eqref{strEx1}, \eqref{R-ex1}, \eqref{KulNom}, \propref{ex1-plosko} and \corref{ex1-cor}, the presented example confirms the statements in \thmref{thm-R1} and \corref{cor-R1}.

By virtue of \eqref{strEx1}, \eqref{n-ex1}, \eqref{D1-T} and \eqref{D1-T-03}, we calculate the components $\Tf_{ijk}= \Tf(e_i,e_j,e_k)$. The nonzero ones of them are determined by the following equalities and their well-known antisymmetries:
\begin{equation}\label{T-ex1}
\begin{array}{l}
\Tf_{013}=\Tf_{031}=\Tf_{024}=\Tf_{042}=1.
\end{array}
\end{equation}

Then, using \eqref{t1} and \eqref{T-ex1}, we calculate $\dot{t}$, $\dot{t}^*$, and $\widehat{\dot{t}}$. The only nonzero one of them is:
\begin{equation}\label{t1-ex1}
\begin{array}{l}
\dot{t}^*(e_0)=4.
\end{array}
\end{equation}

The obtained results in \eqref{T-ex1} and \eqref{t1-ex1} regarding the torsion properties of the studied example confirm the assertion made in \corref{thm:FiT1} in the case of the class $\F_4$.


\end{document}